\documentclass[12pt]{article}
\usepackage[dvips]{graphicx}
\usepackage{amssymb,color}
\usepackage[centertags]{amsmath}
\usepackage{amsfonts}

\def\BBox{\kern  -0.2cm\hbox{\vrule width 0.2cm height 0.2cm}}

\newtheorem{lemma}{Lemma}[section]
\newtheorem{theorem}{Theorem}[section]
\newtheorem{definition}{Definition}[section]
\newtheorem{corollary}{Corollary}[section]

\newtheorem{remark}[theorem]{Remark}

\title{The Graphicahedron}
\author{
Gabriela Araujo-Pardo\thanks{Supported by CONACYT 5737 garaujo@matem.unam.mx}, Maria Del R\'io-Francos\thanks{mad210fcos@gmail.com}, \\
Mariana L\'opez-Dudet\thanks{mdudet17@gmail.com}, 
Deborah Oliveros\thanks{Supported by CONACYT CCDG 50151 dolivero@matem.unam.mx}\\
{\small  Instituto de Matem\'{a}ticas}\\
{\small  Universidad Nacional Aut\'{o}noma de M\'{e}xico, M\'{e}xico}
\\[1ex]
Egon Schulte\thanks{schulte@neu.edu}
\thanks{Supported by NSA-grant 
H98230-07-1-0005}\\
{\small Department of Mathematics}\\
{\small Northeastern University, Boston, USA}\\ }

\begin{document}
\maketitle

\begin{abstract}
The paper describes a construction of abstract polytopes from Cayley graphs of symmetric groups. Given any connected graph $G$ with $p$ vertices and $q$ edges, we associate with $G$ a Cayley graph ${\cal G}(G)$ of the symmetric group $S_p$ and then construct a vertex-transitive simple polytope of rank $q$, the ${\it graphicahedron}$, whose $1$-skeleton (edge graph) is ${\cal G}(G)$. The graphicahedron of a graph $G$ is a generalization of the well-known permutahedron; the latter is obtained when the graph is a path. We also discuss symmetry properties of the graphicahedron and determine its structure when $G$ is small.
\end{abstract}


{\bf Key words.} ~ Permutahedron, abstract polytopes, Cayley Graphs.

{\bf MSC 2000.} ~ Primary: 51M20. Secondary: 05C25, 52B15.


\section{Introduction}
In the present paper we describe an interesting abstract polytope, called the {\em graphicahedron}, which generalizes the well-known permutahedron. The permutahedron $\Pi_{n}$ is a simple convex $n$-polytope in $\mathbb{R}^{n+1}$ that was apparently first investigated by Schoute in 1911 (see \cite{fom,S11,Z95}); it was discovered in Guilbaud \& Rosenstiehl~\cite{guro} in 1963 and given the name ``permutohedron" (or rather ``permuto\`{e}dre", in French). It is the convex hull of all points in $\mathbb{R}^{n+1}$ obtained from $(1,2,\ldots,n+1)$ by permuting the coordinates in all possible ways. In particular, its vertices can be identified with the permutations in the symmetric group $S_{n+1}$ in such a way that two vertices of $\Pi_{n}$ are connected by an edge if and only if the corresponding permutations differ by an adjacent transposition.

Our construction of the graphicahedron builds on Cayley graphs of symmetric groups. Given any connected graph $G$ with $p$ vertices and $q$ edges, we associate with $G$ a Cayley graph ${\cal G}(G)$ of the symmetric group $S_p$ and then construct a vertex-transitive simple polytope of rank $q$, the graphicahedron ${\cal P}_G$, whose $1$-skeleton (edge graph) is ${\cal G}(G)$. The generating set of $S_p$ defining ${\cal G}(G)$ consists of the transpositions associated with the edges of $G$. If $G$ is a path of length $n$, then ${\cal P}_G$ is isomorphic to $\Pi _{n}$. 

The paper is organized as follows. In Sections~\ref{grpol} and \ref{cgra}, we review basic notation and in particular describe the Cayley graph ${\cal G}(G)$ associated with a given graph $G$. Then, in Section~\ref{hedron}, we define the graphicahedron ${\cal P}_G$, establish that it is an abstract polytope, and prove that its $1$-skeleton is the Cayley graph ${\cal G}(G)$. In Section~\ref{group}, we find the structure of the automorphism group of ${\cal P}_G$ and then enumerate the graphicahedra that are regular polytopes. Finally, in Section~\ref{examples}, we describe in more detail the geometric structure of the graphicahedra when $G$ is a path, a cycle, a star graph, or a small graph with up to four edges. We recover the classical permutahedra as the graphicahedra associated with paths, and find other interesting highly-symmetric polytopes, including various locally toroidal polytopes.

\section{Graphs and polytopes}
\label{grpol}
 
We briefly review some important concepts for graphs and polytopes, beginning with graphs. For further basic definitions and terminology on graphs and polytopes the reader is referred to Chartrand \& Lesniak~\cite{CL96} and McMullen \& Schulte~ \cite{McMS02}, respectively. 

Throughout, $G=(V(G),E(G))$ denotes a simple graph, without loops or multiple edges and with vertex-set $V(G)$ and edge-set $E(G)$. We say that $G$ is a $(p,q)$-graph if $|V(G)|=p$ and $|E(G)|=q$. We always assume that $p$ and $q$ are finite. A subgraph of $G$ is called a {\em spanning\/} subgraph of $G$ if it contains all the vertices of $G$. The spanning subgraphs of $G$ are in one-to-one correspondence with the subsets of $E(G)$. In fact, any subset of $E(G)$ uniquely defines a spanning subgraph of $G$, of which it is the full edge set, and vice versa. For any two spanning graphs $H$ and $H'$ of $G$ we write $H\subseteq H'$ if $E(H)\subseteq E(H')$. 
 
An {\it abstract polytope} of rank $n$, or simply an {\it $n$-polytope}, is a partially ordered set (or {\it poset} for short) ${\cal P}$ with a strictly monotone rank function having range $\{-1,0,\ldots,n\}$. The elements of ${\cal P}$ are called {\em faces\/}, or $j$-{\it faces} if their rank is $j$. The faces of ranks 0, 1 or $n-1$ are also called the {\it vertices}, {\it edges} or {\it facets} of $\cal P$, respectively. Moreover, ${\cal P}$ has a smallest face (of rank $-1$) and largest face (of rank $n$), denoted by $F_{-1}$ and $F_n$, respectively; they are the {\it improper faces} of ${\cal P}$. Each {\it flag} (maximal chain) of ${\cal P}$ contains exactly $n+2$ faces. Two flags are said to be {\it adjacent} if they differ in exactly one face; they are {\em $j$-adjacent\/} if this face has rank $j$. In $\cal P$, any two flags $\Phi$ and $\Psi$ can be joined by a sequence of flags $\Phi=\Phi_0,\Phi_1,...,\Phi_k=\Psi$, all containing $\Phi \cap \Psi$, such that any two successive flags $\Phi_{i-1}$ and $\Phi_i$ are adjacent; this property is known as the {\it strong flag-connectedness} of ${\cal P}$. Finally, $\cal P$ has the following {\it homogeneity  property}, often called the {\it diamond condition}:\  whenever $F\leq G$, with ${\rm rank}(F)=j-1$ and ${\rm rank}(G)=j+1$, there are exactly two faces $H$ of rank $j$ such that $F\leq H\leq G$.

Let ${\cal P}$ be an $n$-polytope, and let $0\le k\le n-1$. The {\it $k$-skeleton} of ${\cal P}$ is the poset consisting of all proper faces of ${\cal P}$ of rank at most $k$ (together with the induced partial order).

\section{Cayley Graphs}
\label{cgra}

Given a finite group $\Gamma$ and any subset $\cal T$ of $\Gamma$ consisting of involutions, the {\em Cayley graph of\/} $\Gamma$ {\em with respect to\/} $\cal T$, denoted by ${\cal G}(\Gamma,{\cal T})$, is the graph with vertex-set $\Gamma$ such that two vertices $\gamma_1$ and $\gamma_2$ are adjacent (connected by an edge) if and only if $\gamma_2=\tau\gamma_1$ for some $\tau\in {\cal T}$. Here we slightly abuse standard notation and do not require $\cal T$ to be a generating set of $\Gamma$, although this will usually be the case. Note that ${\cal G}(\Gamma,{\cal T})$ is connected (that is, ${\cal G}(\Gamma,{\cal T})$ is a Cayley graph in standard terminology) if and only if $\cal T$ is a generating set of $\Gamma$. We are primarily interested in Cayley graphs of symmetric groups $S_p$.

Let $G$ be a $(p,q)$-graph with vertex set $V(G):=\{1,\ldots,p\}$ and edge set $E(G)=\{e_1,\ldots,e_q\}$, where $p\geq 1$ and $q\geq 0$ (if $q=0$, then $E(G)=\emptyset$). In most applications, $G$ will be connected. We associate with $G$ a Cayley graph on $S_p$ as follows. If $e = \{i,j\}$ is an edge of $G$ (with vertices $i$ and $j$), define ${\tau}_{e}:= (i\;j)$; this is the transposition in $S_p$ that interchanges $i$ and $j$. Let ${\cal T}_G:=\{\tau_{e_1},\ldots,\tau_{e_q}\}$ denote the set of transpositions determined by the edges of $G$, and let $T_G:=\langle\tau_{e_1},\ldots,\tau_{e_q}\rangle$ denote the subgroup of $S_p$ generated by ${\cal T}_G$. If $G$ is connected, then $T_{G}=S_p$ (see Lemma~\ref{propoconnected} below). More generally, if $K\subseteq E(G)$, then we define ${\cal T}_K:=\{\tau_{e}\mid e\in K\}$ and $T_K:=\langle\tau_{e}\mid e\in K\rangle$. If $K$ happens to be the full edge set of a subgraph $H$ of $G$, then we also write ${\cal T}_H$ or $T_H$ in place of ${\cal T}_K$ and $T_K$; that is, ${\cal T}_{H}:={\cal T}_{E(H)}$ and $T_{H}:=T_{E(H)}$. If $K=\emptyset$, then $T_K$ is the trivial group.

Now the {\em Cayley graph of $G$}, denoted by ${\cal G}(G)$, is the Cayley graph of $S_p$ with respect to ${\cal T}_G$; that is, ${\cal G}(G):={\cal G}(S_p,{\cal T}_G)$. Thus $V({\cal G}(G))=S_p$, and $\{ \gamma_1,\gamma_2 \} \in E({\cal G}(G))$ if and only if $\tau_{e}\gamma_{1} =\gamma_2$ for some $e\in E(G)$. For instance, if $G$ has no edges, then ${\cal G}(G)$ also has no edges; and if $G$ has only one edge, then ${\cal G}(G)$ is a matching on $S_p$. On the other hand, connected graphs are characterized by the following lemma.

\begin{lemma}\label{propoconnected} 
The graph $G$ is connected if and only if ${\cal T}_G$ generates $S_p$. Equivalently, $G$ is connected if and only if ${\cal G}(G)$ is connected.
\end{lemma}

\begin{proof}
Suppose $G$ is connected. We only need to verify that every transposition of $S_p$ is generated by ${\cal T}_G$. Let $(i\;j)$ be any transposition in $S_p$. Since $G$ is connected, there exist a path along edges of $G$ connecting the vertices $i$ and $j$ of $G$. (If $\{i,j\}\in E(G)$, then the path has just one edge.) Suppose the path transverses the edges $f_{1},\ldots,f_{r}$ of $G$, in this order, so that $i$ is the first vertex of $f_{1}$ and $j$ is the last vertex of $f_{r}$. Then 
$(i\;j)=\tau_{f_r}\cdots\tau_{f_2}\tau_{f_1}\tau_{f_2}\cdots\tau_{f_r}$. 

Conversely, suppose $G$ is not connected. If $H$ is a connected component of $G$ and ${\cal T}_H$ is the corresponding set of transpositions, then the subgroup $T_H$ generated by ${\cal T}_H$ is the symmetric group $S_{V(H)}$ on the vertex-set $V(H)$ of $H$ and is a proper subgroup of $S_{p}=S_{V(G)}$. It follows that the subgroup $T_{G}$ generated by${\cal T}_G$ is isomorphic to the direct product of symmetric groups $S_{V(H)}$, where $H$ runs o ver all connected components of $G$. In any case, $T_G$ is a proper subgroup of $S_p$. Thus 
${\cal T}_G$ fails to generate $S_p$.
\end{proof}

Clearly, if $G$ and $G'$ are isomorphic $(p,q)$-graphs then the corresponding Cayley graphs ${\cal G}(G)$ and ${\cal G}(G')$ are also isomorphic. The following lemma summarizes certain basic properties of the Cayley graphs ${\cal G}(H)$ associated with spanning subgraphs $H$ of $G$; they all are subgraphs of ${\cal G}(G)$.

\begin{lemma}\label{remarks} 
(a) A subgraph $H$ of $G$ is a spanning subgraph of $G$ if and only if ${\cal G}(H)$ is a spanning subgraph of ${\cal G}(G)$. \\
(b) If $H$ and $H'$ are spanning subgraphs of $G$ and $H\subseteq H'$, then also ${\cal G}(H)\subseteq {\cal G}(H')$.\\
(c) For any $k$ with $0\leq k\leq q$ we have ${\cal G}(G)=\cup {\cal G}(G_k)$, where  $G_k$ runs over all spanning subgraphs of $G$ with $k$ edges.
\end{lemma}

Figures~\ref{P_2} and \ref{C_3} show the Cayley graphs of the path $G=P_2$ of length $2$ and the cycle $G=C_3$ of length $3$, respectively; now $p=3$ and $q=2$ or $3$, respectively. Here the basic transpositions in ${\cal T}_G$ can be represented by different colors; in fact, color representations of Cayley graphs are quite common and some authors use the term {\em Cayley color graph\/} to emphasize this idea (see \cite{CL96}). 
 
\begin{figure}[h!]
\begin{center}
\scalebox{0.80}{\includegraphics{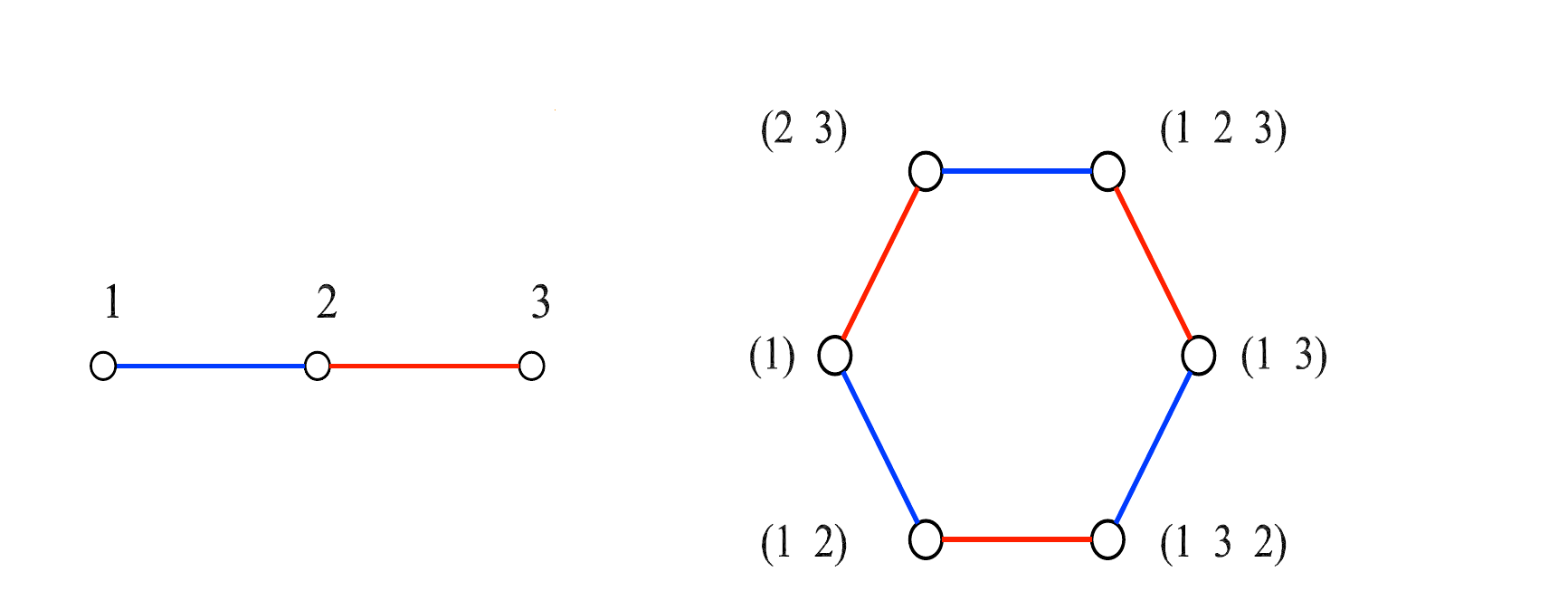}}
\caption{The basic graph $P_2$ and its Cayley color graph ${\cal G}(P_2)
={\cal G}(S_3,\{ \tau_{\{1,2\}},\tau_{\{2,3\}}\})$}. \label{P_2}
\end{center}
\end{figure}

\begin{figure}[h!]
\begin{center}
\scalebox{0.50}{\includegraphics{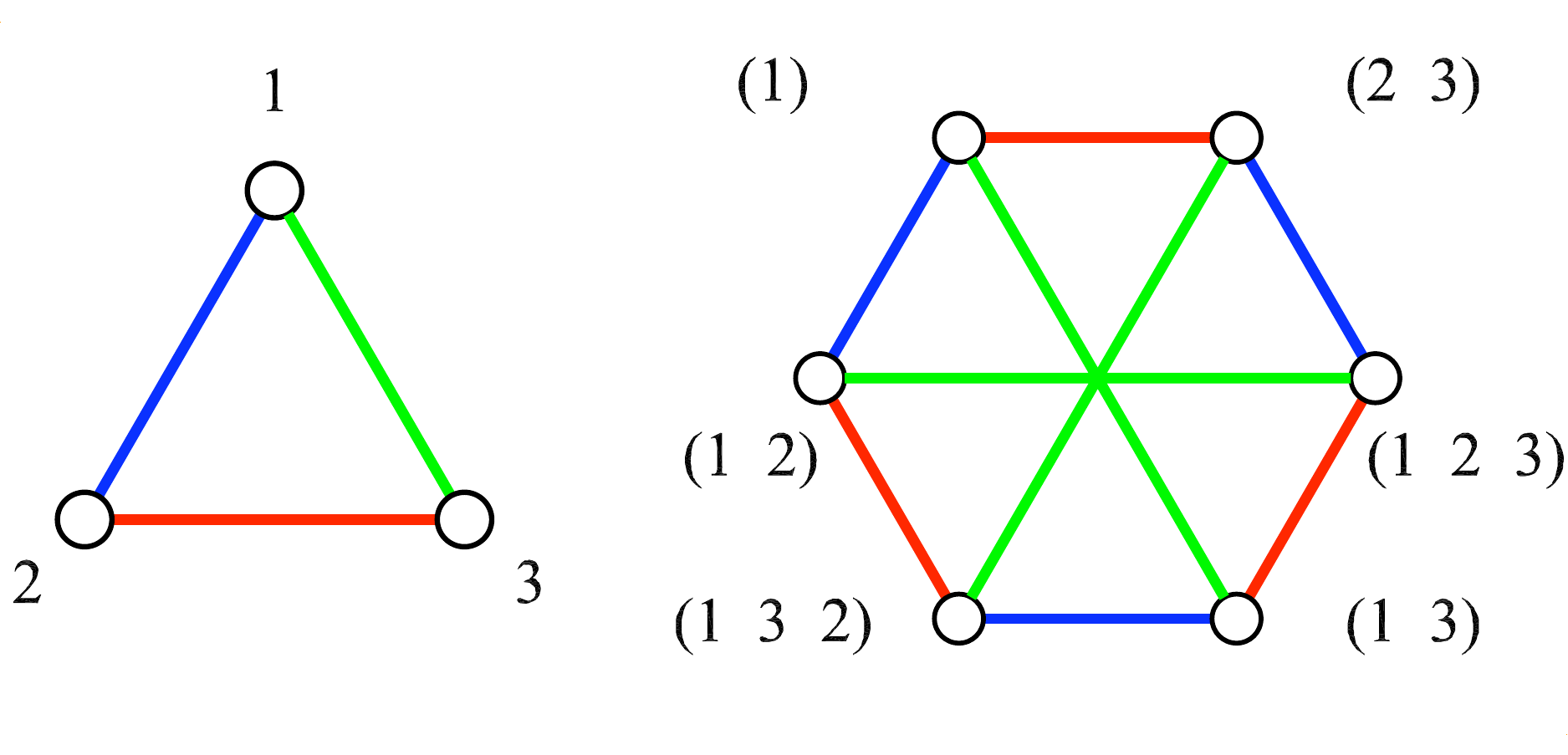}}
\caption{The basic graph $C_3$ and its Cayley color graph 
${\cal G}(C_3)={\cal G}(S_3,\{ \tau_{\{1,2\}},\tau_{\{1,3\}},\tau_{\{2,3\}}\})$}.
\label{C_3}
\end{center}
\end{figure}

The following lemma relates the Cayley graph of the path $P_n$ on $n+1$ vertices to the permutahedron $\Pi_{n}$ in $\mathbb{R}^{n+1}$. For $n=2$, the permutahedron is a hexagon (see Figure~\ref{P_2}).

\smallskip
\begin{lemma} \label{isopermu} 
Let $P_n$ denote the simple path with $n$ edges and $n+1$ vertices. Then the Cayley graph ${\cal G}(P_n)$ of $P_n$ and the $1$-skeleton of the permutahedron $\Pi_{n}$ are isomorphic.
\end{lemma}

\section{The Graphicahedron}
\label{hedron}

Let $G$ be a connected $(p,q)$-graph. In this section, we describe a partially ordered set ${\cal P}_G$, called the {\em graphicahedron}, and show in particular that ${\cal P}_G$ is an abstract polytope of rank $q$. The $i$-faces of ${\cal P}_G$ are given by the connected components of the Cayley graphs of the form ${\cal G}(H)$, where $H$ is a spanning subgraph of $G$ with $i$ edges. 

In defining the graphicahedron it is convenient to initially suppress the (smallest) face of rank $-1$ and concentrate on faces of ranks $0,\ldots,q$. The missing face of rank $-1$ will be appended at the end.

For $i\in I:=\{0,\ldots,q\}$ define
\begin{equation}
\label{facets}
C_i :=  \{ (K,\alpha) \mid K\subseteq E(G), |K| = i,\, \alpha\in S_p \} .
\end{equation} 
Then we have the following equivalence relation on $\bigcup _{i\in I} C_i$.

\begin{definition}\label{equiv} 
Two elements $(K,\alpha),(L,\beta)\in{\bigcup _{i\in I} C_i}$ are said to be equivalent, or 
$(K,\alpha)\sim(L,\beta)$ for short, if and only if $K=L$ and $T_{K}\alpha = T_{L}\beta$ (equality as right cosets in $S_p$).
\end{definition}
By a slight abuse of notation we use the same symbol for both, an element of ${\bigcup _{i\in I} C_i}$ and its equivalence class under $\sim$. 

\begin{definition}\label{order}
Let ${\cal P}_G := {(\bigcup _{i\in I} C_i)}\slash{\sim}$. When $(K,\alpha), (L,\beta) \in {\cal P}_G$, define $(K,\alpha) \leq (L,\beta)$ if and only if $K\subseteq L$ and 
$T_{K}\alpha \subseteq T_{L}\beta$.
\end{definition}

It is straightforward to check that $\leq$ is well-defined and indeed makes ${\cal P}_G$ a partially ordered set. As before, its {\em faces\/} are its elements, $(K,\alpha)$. In particular, ${\cal P}_G$ has a strictly monotone rank function given by
\begin{equation}
{\rm rank}(K,\alpha) := |K| \, .
\end{equation}
Hence $(K,\alpha)$ is an $i$-face of ${\cal P}_G$ if and only if $|K|=i$. In particular, each vertex ($0$-face) of ${\cal P}_G$ is of the form $(\emptyset,\alpha)$ with $\alpha\in S_p$ (and $\emptyset$ as first component). Thus ${\cal P}_G$ has $p!$ vertices. Moreover, there is just one $q$-face, $(E(G),\alpha)$ for any $\alpha$, which is incident with any other face of ${\cal P}_G$. (Recall here that we are still excluding $i=-1$.) 

It is helpful to relate the faces of ${\cal P}_G$ to the connected components of the Cayley graphs of spanning subgraphs of $G$. First recall that the subsets of $E(G)$ are in one-to-one correspondence with the spanning subgraphs of $G$. Thus we may think of a subset $K$ of $E(G)$ as a spanning subgraph, $\widetilde{K}$ (say), of $G$. Then the right coset $T_{K}\alpha$ of the subgroup $T_{K}=T_{\widetilde{K}}$ in $S_p$ involved in Definition~\ref{equiv} is just the connected component of $\alpha$ in the Cayley graph ${\cal G}(\widetilde{K})$, and equivalence of faces in ${\cal P}_G$ is just equality of connected components. In summary, we have

\begin{lemma}
The faces of ${\cal P}_G$ are in one-to-one correspondence with the connected components of the Cayley graphs of spanning subgraphs of $G$. Under this correspondence, if $\widetilde{K}$ denotes the spanning subgraph of $G$ with edge set $K$, then the face $(K,\alpha)$ corresponds to the connected component of $\alpha$ in ${\cal G}(\widetilde{K})$.
\end{lemma}

It is important to point out that the partial order in ${\cal P}_G$ is not in general equivalent to inclusion of connected components of Cayley graphs, or, equivalently, inclusion of right cosets in $S_p$. Clearly, incidence in ${\cal P}_G$ does imply inclusion of cosets; in fact, by definition, if $(K,\alpha) \leq (L,\beta)$ in ${\cal P}_G$, then $T_{K}\alpha \subseteq T_{L}\beta$ in $S_p$. However, the converse is not true in general, since subgroups generated by transpositions often have quite different generating sets of transpositions; there may exist subsets $K$ and $L$, not related by inclusion, but nevertheless satisfying $T_{K}\alpha \subseteq T_{L}\beta$ for some $\alpha,\beta$. On the other hand, the following Remark describes an interesting special case when the converse does hold. 

\begin{remark}
\label{remark}
If $G$ is a tree, then the partial order in ${\cal P}_G$ is equivalent to inclusion of right cosets in $S_p$. In other words, in this case we have $(K,\alpha) \leq (L,\beta)$ in ${\cal P}_G$ if and only if $T_{K}\alpha \subseteq T_{L}\beta$ in $S_p$. 

Here the additional condition that $K\subseteq L$ is already implied by the condition that $T_{K}\alpha \subseteq T_{L}\beta$; that is, the former is not required in the definition of the partial order of ${\cal P}_G$. This follows from the fact that, for a tree $G$, any subgroup of $S_p$ of the form $T_K$ is isomorphic to a direct product of symmetric groups, one for each connected component of the spanning subgraph of $G$ with edge set $K$. In fact, if $T_{K}\alpha \subseteq T_{L}\beta$ in $S_p$, then $T_{L}\beta = T_{L}\alpha$ and hence also $T_{K}\subseteq T_{L}$; but then, in turn, the direct product structure of the subgroups $T_K$ and $T_L$ implies that $K\subseteq L$.
\end{remark}

Resuming our general discussion, let again $G$ be any connected $(p,q)$-graph. Next we investigate the flags of ${\cal P}_G$. First note that, if $(K,\alpha)$ and $(L,\beta)$ are two faces of ${\cal P}_G$ with $(K,\alpha) \leq (L,\beta)$, then $(L,\beta)=(L,\alpha)$ in ${\cal P}_G$, so in representing the latter face we may replace $\beta$ by $\alpha$. In fact, since $(K,\alpha) \leq (L,\beta)$, we have $\alpha\in T_{K}\alpha \subseteq T_{L}\beta$ and hence $T_{L}\alpha = T_{L}\beta$. It follows that any chain of mutually incident faces in ${\cal P}_G$ can be represented in such a way that all second components are the same; in fact, we can always use the second component of the smallest face in the chain.

In particular, a flag $\Phi$ of ${\cal P}_G$ can be described by only two parameters, namely a maximal nested family of subsets of $E(G)$, denoted by 
${\cal K}_{\Phi }:={\cal K}:=\{K_0,K_1,\ldots,K_q\}$, and a single element $\alpha\in S_p$; that is, 
\[ \Phi = ({\cal K},\alpha) := \{(K_0,\alpha),(K_1,\alpha),\ldots,(K_q,\alpha)\} . \]
By a {\em maximal nested\/} family of subsets of a given finite set we mean a flag in the Boolean lattice (power set) associated with this set. Note here that $K_i$ contains exactly $i$ edges, so, in particular, $K_{0}=\emptyset$ and $K_{q}=E(G)$. (Bear in mind that we are still ignoring the $(-1)$-face.)  Thus, as will become clear, a flag is parametrized by its vertex
(in a sense, $\alpha$) and a flag of the vertex-figure at this vertex (in a sense, ${\cal K}$).

Our next lemma is aimed at establishing the strong flag-connectedness of ${\cal P}_G$.

\begin{lemma}\label{families2} 
Let ${\cal K}=\{K_0,K_1,\ldots,K_q\}$ and ${\cal F}=\{F_0,F_1,\ldots,F_q\}$ be two maximal nested families of subsets of $E(G)$. Then there exists a sequence ${\cal K}={\cal K}_0,{\cal K}_1,\ldots,{\cal K}_{s-1},{\cal K}_s={\cal F}$ of maximal nested families of subsets of $E(G)$ such that \\
(a) ${\cal K}_j$ differs from ${\cal K}_{j+1}$ in exactly one element, for all $j =0,\ldots,s-1$;\\
(b) ${\cal K}\cap{\cal F}\subseteq{\cal K}_j$, for all $j=0,\ldots,s$.
\end{lemma}
\begin{proof}
One possible way to establish the lemma is to appeal to the strong flag-connectedness of the face-lattice of the $(q-1)$-simplex; the latter is isomorphic to the Boolean lattice on $E(G)=\{e_1,\ldots,e_q\}$. Viewed in this setting, the two families ${\cal K}$ and ${\cal F}$ are flags of the $(q-1)$-simplex and can be joined by a sequence of successively adjacent flags of the $(q-1)$-simplex all containing ${\cal K}\cap{\cal F}$. 
\end{proof}

Next we settle the strong flag-connectedness of ${\cal P}_G$. We will see that, in essence, the previous lemma corresponds to the strong flag-connectedness of the vertex-figures of ${\cal P}_G$.

\begin{lemma}\label{flagconn}
The poset ${\cal P}_G$ is strongly flag-connected. In other words, if $\Phi$ and $\Psi$ are two flags of ${\cal P}_G$, then there exists a sequence of successively adjacent flags $\Phi=\Phi_0,\Phi_1,\ldots,\Phi_{s-1},\Phi_s=\Psi$ of ${\cal P}_G$ such that $\Phi\cap\Psi\subseteq\Phi_i$ for all $i=0,\ldots,s$.
\end{lemma}

\begin{proof} 
Let $\Phi=\{(K_0,\alpha),\ldots,(K_q,\alpha)\}$ and $\Psi=\{(F_0,\beta),\ldots,(F_q,\beta)\}$ be two flags of ${\cal P}_G$, where $\alpha,\beta\in S_p$ and ${\cal K}:=\{K_0,\ldots,K_{q-1}\}$ and ${\cal F}:=\{F_0,\ldots,F_{q-1}\}$ are maximal nested families of subsets of $E(G)$, respectively; that is, $\Phi=({\cal K},\alpha)$ and $\Psi=({\cal F},\beta)$. Let $J$ denote the set of suffices $j$ with $(K_{j},\alpha)=(F_{j},\beta)$; then $J$ consists of the positions $j$ where $\Phi$ and $\Psi$ agree, and $J\neq\emptyset$ since $q\in J$. Let $m$ denote the smallest suffix in $J$.

First we settle the case that $\Phi$ and $\Psi$ share a $0$-face (vertex), or equivalently, that $\alpha=\beta$, or that $m=0$. If $\alpha=\beta$, we can appeal to Lemma~\ref{families2} to obtain a sequence ${\cal K}={\cal K}_0,{\cal K}_1,\ldots,{\cal K}_{t-1}, {\cal K}_t={\cal F}$ of maximal nested families of subsets of $E(G)$, such that any two consecutive families differ in exactly one element and all families contain ${\cal K}\cap{\cal F}$. Now define $\Phi_{i}:=({\cal K}_i,\alpha)$ for $i=0,\ldots,t$. Then the sequence of flags $\Phi_0,\Phi_1,\ldots,\Phi_{s-1},\Phi_s$ of ${\cal P}_G$ has all the required properties. Thus the vertex-figure of ${\cal P}_G$ at its vertex $(K_0,\alpha)=(F_0,\beta)$ (in fact, at any of its vertices) is strongly flag-connected. 

The proof of the general case rests on the connectedness of $G$ (recall that the latter had been assumed from the start). Thus the underlying Cayley graph ${\cal G}(G)$ is connected and the transpositions in ${\cal T}_G$ generate the full symmetric group $S_p$ (that is, $T_G=S_p$). Then, since $(K_m,\alpha)=(F_m,\beta)$ in ${\cal P}_G$ (recall the definition of $m$), there exist edges $f_1,\ldots,f_s$ in $K_{m}\,(=F_{m})$ and corresponding transpositions $\tau_{f_1},\ldots,\tau_{f_s}$ in ${\cal T}_{K_{m}}$ such that $\beta=\tau_{f_s}\cdots\tau_{f_2}\tau_{f_1}\alpha$, provided $m\geq 1$. Here it is convenient to also allow $m=0$, meaning that $\beta=\alpha$ and $s=0$ (that is, no transpositions occur). If $m=1$, then $\Phi$ and $\Psi$ agree in their $1$-face but not their $0$-face; since $K_1$ consists of a single edge, we must have $s=1$, $K_{1}=\{f_1\}$, and $\beta=\tau_{f_1}\alpha$. Recall from Lemma~\ref{propoconnected} that $T_{K_m}=S_p$ if $m=q$.

We now proceed by induction on $s$. The case $s=0$ (that is, $\beta=\alpha$) has already been settled. Now let $s\geq 1$ and hence also $m\geq 1$ (possibly $m=q$). The key inductive step is to properly join the initial flag $\Phi=({\cal K},\alpha)$ to a suitable new flag $\Lambda=({\cal L},\gamma)$ associated with a maximal nested family ${\cal L}$ of subsets of $E(G)$ and the element $\gamma:=\tau_{f_1}\alpha$ of $S_p$. At the end, since $\beta=\tau_{f_s}\cdots\tau_{f_2}\gamma$, the inductive hypothesis for $s-1$, applied to $\Lambda$ and $\Psi$, will yield a sequence of flags joining $\Lambda$ to $\Psi$. Finally, then, the two sequences of flags connecting $\Phi$ to $\Lambda$ and $\Lambda$ to $\Psi$, respectively, can be concatenated to obtain a sequence of flags connecting $\Phi$ to $\Psi$.

It remains to construct $\Lambda$. To begin with, observe that in ${\cal P}_G$ we have
\[ (K_0,\alpha),(F_0,\tau_{f_{1}}\alpha)\,\leq\, (\{f_1\},\alpha)\,\leq\, (K_m,\alpha)
= (K_m,\beta), \] 
since both $\tau_{f_{1}}\alpha\in\langle\tau_{f_{1}}\rangle\alpha$ and 
$\tau_{f_{1}}\alpha = (\tau_{f_s}\cdots\tau_{f_2})^{-1}\beta \in T_{{K_m}}\beta = T_{{K_m}}\alpha$, respectively. Now we can proceed in three steps as follows. First, by our observation, we may choose a flag $\Lambda'$ (say) which contains $(K_0,\alpha)$ as its $0$-face and $(\{f_1\},\alpha)$ as its $1$-face and also includes the entire subset $\Phi\cap\Psi$; this is possible even if $m=1$. Then, since $\Phi$ and $\Lambda'$ share a $0$-face, we can join them by a sequence of successively adjacent flags, all containing $\Phi\cap\Lambda'$ and therefore also $\Phi\cap\Psi$. Second, substitute the $0$-face of $\Lambda'$ by $(F_0,\tau_{f_{1}}\alpha)$ to obtain a new flag $\Lambda''$ $0$-adjacent to $\Lambda'$, and then append $\Lambda''$ to the existing flag sequence. Third, choose any maximal nested family ${\cal L}=\{L_0,\ldots,L_q\}$ of sets of edges of $G$ such that $L_{1}=\{f_1\}$ and $L_{j}=K_{j}=F_{j}$ for all $j\in J$; this is possible, since $f_{1}\in K_m$. 
Define
\[ \Lambda := ({\cal L},\tau_{f_{1}}\alpha)
=\{(L_0,\tau_{f_{1}}\alpha),\ldots,(L_q,\tau_{f_{1}}\alpha)\} .\]
Then $\Lambda$ shares a vertex with $\Lambda''$ and contains $\Phi\cap\Psi$; note for the latter that $f_1\in K_j=L_j$ for all $j\in J$. Finally, since $\Lambda''$ and $\Lambda$ share a common $0$-face, we can further extend the already existing flag sequence by a sequence of successively adjacent flags, which all contain $\Lambda''\cap\Lambda$ and hence also $\Phi\cap\Psi$, and which joins $\Lambda''$ to $\Lambda$. We have now joined 
$\Phi$ to $\Lambda$, and we are done. 
\end{proof}

We also need to prove that ${\cal P}_G$ satisfies the diamond condition. 

\begin{lemma}
\label{diamond}
For each $i=1,\ldots,n-1$, given an $(i-1)$-face $(K,\alpha)$ and an $(i+1)$-face $(L,\beta)$ of ${\cal P}_G$ such that $(K,\alpha)\le(L,\beta)$, there exist exactly two $i$-faces $(J,\gamma)$ of ${\cal P}_G$ such that $(K,\alpha)\le (J,\gamma) \le(L,\beta)$. Similarly, given a $1$-face $(L,\beta)$ of ${\cal P}_G$, there exist exactly two $0$-faces $(J,\gamma)$ of ${\cal P}_G$ such that $(J,\gamma) \le(L,\beta)$.
\end{lemma}
\begin{proof}
For the first part, suppose $(K,\alpha)\le (J,\gamma) \le(L,\beta)$ in ${\cal P}_G$. Then, as explained earlier, we may assume that $\gamma=\alpha=\beta$. Moreover, $K\subset J\subset L$ and $K,J,L$ contain exactly $i-1$, $i$ or $i+1$ elements, respectively. This leaves precisely two possible choices for $J$. Conversely, any such choice determines an $i$-face $(J,\alpha)$ such that $(K,\alpha)\le (J,\alpha) \le(L,\alpha)$.

For the second part, suppose $(J,\gamma) \le(L,\beta)$ in ${\cal P}_G$. Now $J=\emptyset$ and $L$ has only one element. If $L=\{e\}$ (say), then necessarily 
$\gamma \in T_{L}\beta=\langle\tau_e\rangle\beta = \{\beta,\tau_e\beta\}$. This leaves only two choices for $\gamma$, namely $\beta$ or $\tau_e\beta$. Conversely, either of these choices yields a $0$-face $(J,\gamma)$ with $(J,\gamma) \le(L,\beta)$.
\end{proof}

Finally, then, we need to append a face of rank $-1$ as smallest face of ${\cal P}_G$. In particular, from now on, ${\cal P}_G$ will denote the extended partially ordered set.

The symmetric group $S_p$ acts faithfully on ${\cal P}_G$ as a group of polytope automorphisms. In fact, every $\gamma\in S_p$ determines an element of $\Gamma({\cal P}_G)$, again denoted by $\gamma$, given by 
\begin{equation}\label{aut}
(K,\alpha)\rightarrow (K,\alpha\gamma) \qquad (K\subseteq E(G),\, \alpha\in S_p).
\end{equation}
(Note here that $\gamma$ is realized by right multiplication on the second component, not by left multiplication. Left multiplication would not in general lead to an incidence preserving mapping of ${\cal P}_G$.) It is straightforward that the group $S_p$ of all such polytope automorphisms acts simply transitively on the vertices of 
${\cal P}_G$. In general, however, $S_p$ is only a proper subgroup of the full automorphism group of ${\cal P}_G$. The precise relationship is clarified in Theorem~\ref{graphsaut}.

Call an abstract $q$-polytope {\em simple\/} if all its vertex-figures are isomorphic to the $(q-1)$-simplex. In ${\cal P}_G$, the flags which contain a given vertex $(\emptyset,\alpha)$ are all of the form $({\cal K},\alpha)$, where ${\cal K}$ is a maximal nested family of subsets of $E(G)$. It follows that ${\cal P}_G$ is simple. Moreover, if $(K,\alpha)$ is a facet (that is, $(q-1)$-face) of ${\cal P}_G$ containing $(\emptyset,\alpha)$, then $|K|=q-1$ and $K$ is obtained from $E(G)$ by removing a single element. This gives exactly $q$ choices for $K$ and hence $q$ choices of facets containing 
$(\emptyset ,\alpha)$. Therefore, since $S_p$ acts vertex-transitively on ${\cal P}_G$, there are exactly $q$ types of facets in ${\cal P}_G$, each type occurring at every vertex of ${\cal P}_G$. Note here that it can happen that two distinct facets $(K,\alpha)$ and $(L,\alpha)$ containing a given vertex $(\emptyset,\alpha)$ determine the same coset in $S_p$, that is, $T_K\alpha=T_L\alpha$ (the facets still are distinct because $K$ and $L$ are distinct).

In summary, we have established the following theorem. 

\begin{theorem} \label{theo1} 
Let $G$ be a connected $(p,q)$-graph. Then ${\cal P}_G$ is a simple abstract polytope of rank $q$ with $p!$ vertices and $p!q!$ flags. Moreover, the symmetric group $S_p$ acts simply transitively on the vertices of ${\cal P}_G$.
\end {theorem} 

Note that ${\cal P}_G$ is a $0$- or $1$-polytope, respectively, if $G$ is the trivial graph (with a single vertex and no edge) or a graph with a single edge (and two vertices). If $G$ has exactly two edges, then ${\cal P}_G$ is a hexagon (see also Theorem~\ref{theo3}).

The polytope ${\cal P}_G$ of Theorem~\ref{theo1} is called the {\em graphicahedron associated with\/} $G$, or simply the {\em $G$-graphicahedron\/}. The following theorem lies at the heart of our construction and is largely responsible for the use of the term ``graphicahedron".

\begin{theorem} \label{theo2}
Let $G$ be a connected $(p,q)$-graph. Then the $1$-skeleton of ${\cal P}_G$ is isomorphic to the Cayley graph ${\cal G}(G)$ associated with $G$. 
\end{theorem}

\begin{proof}
Clearly, the vertices of ${\cal P}_G$ can be identified with the elements of $S_p$, via $(\emptyset,\alpha) \rightarrow \alpha$. Moreover, in ${\cal P}_G$, two vertices $(\emptyset,\alpha)$ and $(\emptyset,\beta)$ are incident with a common $1$-face $(\{e\},\gamma)$ (say) if and only if $\langle\tau_e\rangle\gamma = \{\alpha,\beta\}$; that is, if and only if $\beta=\tau_{e}\alpha$. Thus, adjacency in the $1$-skeleton of ${\cal P}_G$ corresponds precisely to adjacency in the Cayley graph ${\cal G}(G)$.
\end{proof}

Cayley graphs are a fruitful source for the construction of polytopes. For two quite different applications of Cayley graph techniques to polytopes see also Monson-Weiss~\cite{mow} and Pellicer~\cite{pel}.

\section{The group of the graphicahedron}
\label{group}

The automorphism group $\Gamma({\cal P}_{G})$ of the graphicahedron ${\cal P}_G$ will usually be larger than the underlying symmetric group $S_p$. In fact, this happens precisely when $G$ has non-trivial graph automorphisms (symmetries). Let $\Gamma(G)$ denote the group of graph automorphisms of $G$. Note that, if $G$ is connected, then $\Gamma(G)$ is faithfully represented by its actions on the vertices of $G$, and on the edges of $G$ provided $G$ has at least two edges or no edge. 

\begin{theorem}\label{graphsaut}
Let $G$ be a connected $(p,q)$-graph, and let $q\neq 1$. Then $\Gamma({\cal P}_{G})=S_{p}\ltimes\Gamma(G)$.
\end{theorem}

\begin{proof}
First we show that $\Gamma({\cal P}_G)$ contains a subgroup isomorphic to a semi-direct product $S_{p}\ltimes\Gamma(G)$. We already know that $S_p$ can be viewed as a subgroup of $\Gamma({\cal P}_{G})$ acting by right multiplication on the second component of the faces of ${\cal P}_G$. 

Now consider $\Gamma(G)$. Recall that $q\geq 2$. Every graph automorphism $\kappa$ of $G$ is an incidence preserving permutation of the vertices and edges of $G$ and hence determines two mappings. First, $\kappa$ clearly acts on the subsets of $E(G)$ via its action on the $q$ edges of $G$. Second, $\kappa$ induces a group automorphism on $S_p$ via its action on the $p$ vertices of $G$, namely through conjugation in $S_p$ by $\kappa$; in particular, if $e$ is an edge of $G$, then $\kappa(e)$ is also an edge of $G$ and $\kappa\tau_{e}\kappa^{-1}=\tau_{\kappa(e)}$ in $S_p$. Hence, $\kappa$ also determines the following mapping on ${\cal P}_G$, again denoted by $\kappa$:
\begin{equation}\label{semd}
(K,\alpha) \rightarrow (\kappa(K),\alpha^{\kappa}) 
\qquad (K\subseteq E(G),\, \alpha\in S_p), 
\end{equation}
with $\alpha^{\kappa}:=\kappa\alpha\kappa^{-1}$. We need to show that this is indeed a polytope automorphism of ${\cal P}_G$. 

First note that in $S_p$ we have  
\[ T_{\kappa(K)} = \langle\tau_{\kappa(e)}\mid e\in K\rangle
=\kappa\langle\tau_{e}\mid e\in K\rangle\kappa^{-1} 
=\kappa T_K \kappa^{-1} \]
and hence also 
\[ T_{\kappa(K)}\alpha^{\kappa}
=\kappa T_K \kappa^{-1}\kappa\alpha\kappa^{-1}
=\kappa T_K\alpha \kappa^{-1}. \]
Therefore, if $K\subseteq L$ and $\alpha,\beta\in S_p$, then $T_{K}\alpha\subseteq T_{L}\beta$ if and only if $T_{\kappa(K)}\alpha^{\kappa}\subseteq T_{\kappa(L)}\beta^{\kappa}$. In ${\cal P}_G$, then this says that $(K,\alpha)\leq (L,\beta)$ if and only if $(\kappa(K),\alpha^{\kappa})\leq(\kappa(L),\beta^{\kappa})$. Thus $\kappa$ is an incidence preserving bijection, that is, a polytope automorphism of ${\cal P}_G$.

By slight abuse of notation, we let $\Gamma(G)$ also denote the subgroup of $\Gamma({\cal P}_G)$ consisting of all polytope automorphism $\kappa$ of ${\cal P}_G$ derived in this way from graph automorphisms of $G$. Note here that this is a faithful copy of the group of graph automorphisms of $G$; in fact, if every face $(K,\alpha)$ in (\ref{semd}) is fixed under $\kappa$, then the corresponding graph automorphism must map every subset $K$ of $E(G)$ to itself and so must be the identity mapping since $G$ is connected. 

At this point we know that both $S_p$ and $\Gamma(G)$ can be viewed naturally as subgroups of $\Gamma({\cal P}_G)$. These subgroups intersect only trivially, since every polytope automorphism in $S_p$ leaves the first component $K$ of every face $(K,\alpha)$ unchanged, while only the trivial polytope automorphism in $\Gamma(G)$ has this property. Moreover, the subgroup $S_p$ is invariant under conjugation in $\Gamma({\cal P}_G)$ by a polytope automorphisms $\kappa$ in $\Gamma(G)$. In fact, if $\gamma\in S_p$, then the polytope automorphism $\kappa\gamma\kappa^{-1}$ takes the face $(K,\alpha)$ of ${\cal P}_G$ to
\[ (K,(\alpha^{\kappa^{-1}}\gamma)^{\kappa})
= (K,\alpha\gamma^{\kappa}),\]
and hence must coincide with the polytope automorphism determined by the element $\gamma^{\kappa}$ of the underlying symmetric group $S_p$; as a graph automorphism, $\kappa$ is a permutation of the vertices, so $\gamma^{\kappa}$ really is an element of $S_p$. Thus $S_p$ is normalized by $\Gamma(G)$, and
we have a semi-direct product $S_{p}\ltimes\Gamma(G)$ realized as a subgroup of $\Gamma({\cal P}_G)$. 

It remains to show that every polytope automorphism of ${\cal P}_G$ lies in this semi-direct product. Suppose $\rho$ is a polytope automorphism of ${\cal P}_G$. 
Then $\rho$ takes vertices of ${\cal P}_G$ to vertices of ${\cal P}_G$, and in particular takes $(\emptyset,\epsilon)$ to a vertex $(\emptyset,\gamma)$ (say) with 
$\gamma\in S_p$. (As before, $\epsilon$ denotes the identity element in $S_p$.)
Under $\gamma^{-1}$, viewed as a polytope automorphism of ${\cal P}_G$, this new vertex is mapped back to $(\emptyset,\epsilon)$. It follows that the polytope automorphism $\kappa:=\gamma^{-1}\rho$ maps $(\emptyset,\epsilon)$ to itself and hence yields a polytope automorphism of the vertex-figure of ${\cal P}_G$ at this vertex. Our goal is to show that $\kappa$ is contained in the subgroup $\Gamma(G)$ of $\Gamma({\cal P}_G)$. Then this would imply
\[ \rho = \gamma\kappa \in S_{p}\cdot\Gamma(G) ,\]
as required.

We know that a polytope automorphism is completely determined by its effect on a single flag; this is implied by the strong flag-connectedness. Now pick any flag of ${\cal P}_G$ of the form $\Phi:=({\cal K},\epsilon)$, where, as usual, ${\cal K}=\{\emptyset,K_1,\ldots,K_q\}$ is a maximal nested family of subsets of $E(G)$. Note that $(\emptyset,\epsilon)$ is the vertex contained in $\Phi$. Since $\kappa$ preserves the vertex-figure of ${\cal P}_G$ at $(\emptyset,\epsilon)$, the image of $\Phi$ under $\kappa$ is a flag $\Lambda:=({\cal L},\epsilon)$, where ${\cal L}=\{\emptyset,L_1,\ldots,L_q\}$ is a maximal nested family of subsets of $E(G)$. In particular,
$\kappa$ maps $(K_j,\epsilon)$ to $(L_{j},\epsilon)$ for each $j=1,\ldots,q$. 

Suppose the edges of $G$ have been labeled $f_1,\ldots,f_q$ and $g_1,\ldots,g_q$, so that $K_{j}=\{f_1,\ldots,f_j\}$ and $L_{j}=\{g_1,\ldots,g_j\}$ for each $j$. Define the permutation $\kappa'$ of $E(G)$ by $\kappa'(f_j)=g_j$ for each $j$. Then $\kappa'(K_j)=L_j$ for each $j$, so in particular, $\kappa$ maps $(K_j,\epsilon)$ to $(\kappa'(K_j),\epsilon)$ for each $j$. Suppose we already know that $\kappa'$ is the edge permutation of $G$ determined by a graph automorphism of $G$, again denoted by $\kappa'$. Then (\ref{semd}) shows that $\kappa'$, viewed as a polytope automorphism, maps the face $(K_j,\epsilon)$ of ${\cal P}_G$ to 
\[ (\kappa'(K_j),\epsilon^{\kappa'})
=(\kappa'(K_j),\epsilon)=(L_j,\epsilon), \]
for each $j$, and therefore maps $\Phi$ to $\Lambda$. Since polytope automorphisms are uniquely determined by their effect on a single flag, this would imply that $\kappa=\kappa'\in\Gamma(G)$, as desired.

Therefore we must prove that $\kappa'$ is the edge permutation of $G$ given by a graph automorphism of $G$. Here it remains to find the action of the latter on the vertices of $G$. We show that $\kappa'$ determines a vertex permutation $\kappa''$ (say) of $G$, such that $\kappa'$ and $\kappa''$ together yield an incidence preserving bijection of $G$, that is, a graph automorphism of $G$. 

First note that a graph automorphism of a connected graph is completely determined by its effect on the vertices of valency at least $2$ and their neighbors of valency $1$. Thus we may ignore vertices of $G$ of valency $1$, since we know the effect on the edges that contain them. Let $j$ be a vertex of $G$ of valency at least 2, and let $E(j)$ be the set of edges of $G$ that contain vertex $j$. It suffices to show that, given any vertex $j$ of $G$ of valency at least $2$, the image of $E(j)$ under $\kappa'$ is contained in a set $E(j'')$ for some vertex $j''$ of $G$ (which then trivially must have valency at least $2$). In fact, once this has been proved, we can simply define the desired vertex permutation $\kappa''$ by $\kappa''(j)=j''$ for any vertex $j$ of $G$ of valency at least $2$; this determines $\kappa''$ completely, and we are done. 

The key for the final step is to remember that $\kappa'$ was derived from the above polytope automorphism $\kappa$ of the underlying graphicahedron ${\cal P}_G$.  We know that the face structure of ${\cal P}_G$ is completely determined by the structure of $G$. For example, the $2$-face $(\{e,f\},\epsilon)$ of ${\cal P}_G$ is a hexagon or square, depending on whether or not $e,f$ are edges of $G$ with a common vertex; note here that 
$T_{\{e,f\}}= \langle\tau_e,\tau_f\rangle = S_3$ or $C_{2}\times C_{2}$, according as $e$ and $f$ share a vertex or not. On the other hand, polytope automorphisms preserve isomorphism types of faces, so $\kappa$ must necessarily map square $2$-faces to square $2$-faces and hexagonal $2$-faces to hexagonal $2$-faces; in particular, $T_{\{\kappa'(e),\kappa'(f)\}}\cong T_{\{e,f\}}$. Hence, the edges $e$ and $f$ of $G$ share a common vertex in $G$ if and only if their images $\kappa'(e)$ and $\kappa'(f)$ share a common vertex in $G$.

Finally, then, let $j$ be any vertex of $G$ of valency $s\geq 2$, and let $E(j)$ be as above. When $s=2$ we are done by the previous argument; in fact, if $e$ and $f$ are the two edges of $E(j)$ (sharing vertex $j$), then $\kappa'(e)$ and $\kappa'(f)$ also share a vertex, $j''$, and $\kappa'(E(j))\subseteq E(j'')$. Suppose $s\geq 3$. 
We now consider $3$-faces of ${\cal P}_G$. If $e,f,g$ are any edges in $E(j)$ (all sharing the vertex $j$), then we must have
$T_{\{e,f,g\}}= \langle\tau_e,\tau_f,\tau_g\rangle = S_4$; in fact, if $k,l,m$ (say) are the end vertices of $e,f,g$ distinct from $j$, respectively, then 
\[ \tau_e = (j\;k),\;\,  \tau_f = (j\;l),\;\, \tau_g = (j\;m), \]
and these transpositions generate all permutations of $j,k,l,m$. In particular, the $3$-face $(\{e,f,g\},\epsilon)$ of ${\cal P}_G$ must have $24$ vertices, so the same must be true for its image $(\{\kappa'(e),\kappa'(f),\kappa'(g)\},\epsilon)$ under $\kappa$. Hence 
$T_{\{\kappa'(e),\kappa'(f),\kappa'(g)\}}
= \langle\tau_{\kappa'(e)},\tau_{\kappa'(f)},\tau_{\kappa'(g)}\rangle$ 
must also have order $24$. Since we already know from our previous considerations that any two of $\kappa'(e),\kappa'(f),\kappa'(g)$ share a common vertex in $G$,
this now forces all three to have a vertex in common (otherwise we would only obtain a subgroup $S_3$). It follows that $\kappa'$ maps any three edges $e,f,g$ in $E(j)$ to three concurrent edges. However, if any three edges in $E(j)$ are mapped to three concurrent edges, then the entire set $E(j)$ is mapped to a set of edges of $G$ all sharing a common vertex, $j''$. Thus $\kappa'(E(j))\subseteq E(j'')$, and the proof is complete.
\end{proof} 

Note that Theorem~\ref{graphsaut} fails if $G$ is a connected graph with only one edge. In fact, in this case ${\cal P}_G$ is a $1$-polytope with automorphism group $S_2$, while $\Gamma(G)$ also has order $2$ (but has a trivial action on the edge set of $G$). 

The following theorem characterizes the graphicahedra which are regular polytopes. By $P_n$ and $C_n$, respectively, we denote the path or cycle of length $n$ (with $n$ edges), and by $K_{1,n}$ the star graph with $n$ edges emanating from a central vertex. Note that $K_{1,n}=P_n$ if $n=0,1,2$, so in particular $K_{1,0}$ is the trivial graph (with a single vertex and no edge).

\begin{corollary}\label{allsq}
Let $G$ be a connected $(p,q)$-graph. Then ${\cal P}_G$ is regular if and only if $G=C_3$ or $G=K_{1,q}$ with $q\geq 0$.
\end{corollary}

\begin{proof}
If $q=0$ or $1$, then ${\cal P}_G$ is a $0$- or $1$-polytope and hence is regular; in this case $G=K_{1,q}$. Hence we may assume that $q\geq 2$.

A polytope is regular if and only if the order of its automorphism group is the same as the number of its flags (see \cite[2A5]{McMS02}). By Theorem~\ref{graphsaut}, $\Gamma({\cal P}_{G})=S_{p}\ltimes\Gamma(G)$; and by Theorem~\ref{theo1}, ${\cal P}_G$ has $p!q!$ flags. Hence ${\cal P}_G$ is regular if and only if $\Gamma(G)$ has order $q!$. Since graph automorphisms are completely determined by their effect on the edges of the graph, the latter is equivalent to saying that $\Gamma(G)=S_q$, or, that every permutation of edges of $G$ arises from a graph automorphism of $G$. 

It remains to enumerate the graphs $G$ with $q$ edges and with $\Gamma(G)=S_q$. Now suppose $G$ is a graph with $q$ edges and with $\Gamma(G)=S_q$.

First we show that $G$ is a tree unless $G=C_3$. Suppose $C$ is a cycle of smallest length, $k$ (say), in $G$. Let $f_1,\ldots,f_k$ denote the edges of $C$, in cyclic order. Now,
since $\Gamma(G)=S_q$, if $f$ is any edge of $G$, then the edges $f,f_2,\ldots,f_k$ must also form a $k$-cycle, $C_f$ (say), of $G$; in fact, any edge permutation that maps $f_1$ to $f$ and fixes each $f_j$ for $j\geq 2$ must necessarily come from a graph automorphism that maps $C$ to a $k$-cycle of $G$. But then $C_{f}=C$ and hence $f=f_{1}$, since the two cycles share all but one edge (double edges are not permitted). Hence the edges of $C$ comprise all the edges of $G$, and hence $G$ itself is a cycle of length $k=q$. The automorphism group of a cycle of length $q$ is a dihedral group of order $q$, and is isomorphic to $S_q$ only if $q=3$. This leaves $C_3$ as the only possibility. Thus $G$ is acyclic unless $G=C_3$.

If $G$ has only vertices of valency at most $2$, then $G$ is the path of length $q$. Since the automorphism group of a path of any length has order $2$, this forces $q=2$, that is, $G=P_{2}=K_{1,2}$.

Now suppose $G$ has a vertex $j$ (say) of valency at least $3$. Let $e,f,g$ be edges of $G$ containing $j$. If $e'$ is any edge of $G$, then $e'$ must also contain $j$; in fact, any edge permutation that maps $e$ to $e'$ and fixes $f$ and $g$ must necessarily come from a graph automorphism that fixes the vertex $j$ and hence takes $e$ to an edge, $e'$, that also contains $j$. Thus every edge of $G$ contains the vertex $j$ of $G$. It follows that $G=K_{1,q}$.

In summary, we have shown that $G$ must be a graph $C_3$ or $K_{1,q}$ with $q\geq 2$. Conversely, each such graph $G$ has $S_3$ or $S_q$, respectively, as its automorphism group, so that ${\cal P}_G$ must be regular. 
\end{proof}

\section{Examples}
\label{examples}
In this section we discuss in more detail the geometric structure of the graphicahedron ${\cal P}_G$ for some specific choices of $(p,q)$-graph $G$. Recall that $G$ must be connected.

We begin with the path $G=P_n$ of length $n$, for which $p=n+1$ and $q=n$. As mentioned earlier, the graphicahedron associated with a given graph $G$ generalizes the permutahedron $\Pi_{n}$.  The following theorem shows that ${\cal P}_{G}$ is in fact isomorphic to $\Pi_{n}$ if $G=P_n$. 

\begin{theorem}
\label{theo3}
The graphicahedron ${\cal P}_{P_n}$ associated with a path $P_n$ of length $n$ is isomorphic to the face-lattice of the permutahedron $\Pi_{n}$ of rank $n$.
\end{theorem}

\begin{proof}
If we knew that ${\cal P}_{P_n}$ could be realized as a convex $n$-polytope in some euclidean space, then, bearing in mind Lemma~\ref{isopermu} and Theorem~\ref{theo2}, we could simply appeal to the well-known fact that a simple convex polytope of a given dimension is uniquely determined up to isomorphism by its $1$-skeleton (see \cite{BlindMan87,kalai88}). However, we do not a priori have such a realization and must proceed in a different manner. We outline the proof.

We use some basic Coxeter group theory for the symmetric group $S_{n+1}$, or, equivalently, the symmetry group of the regular $n$-simplex. The permutahedron $\Pi_n$ was defined earlier as the convex hull of the $(n+1)!$ points (in a hyperplane) in $\mathbb{R}^{n+1}$ obtained from $(1,2,\ldots,n+1)$ by permuting the coordinates in all possible ways. However, $\Pi_n$ has other equivalent realizations. In fact, $\Pi_n$ can also be obtained from Wythoff's construction applied to an interior point $x$ of the standard fundamental simplex for the symmetry group $S_{n+1}$ of the regular $n$-simplex (see Coxeter~\cite{RP}); in other words, $\Pi_n$ is the convex hull of all images of $x$ under this group. Since the (interior) walls of the fundamental simplex are in one-to-one correspondence with the standard generating system of adjacent transpositions $(1\;2), (2\;3),\ldots,(n\;n+1)$ of $S_{n+1}$, it is immediately clear that the $1$-skeleton of $\Pi_n$ is the Cayley graph of $S_{n+1}$. In the notation of \cite[\S 11.6-7]{RP}, $\Pi_n$ may be represented by a Coxeter diagram $A_n$ in which all nodes are ringed, and the structure of its faces can be read off this diagram. Combinatorially, the boundary complex of $\Pi_n$ is the dual of the barycentric subdivision of the boundary complex of the regular $n$-simplex. Comparison with the graphicahedron ${\cal P}_{P_n}$ then shows that the faces of $\Pi_n$ may be viewed as geometric realizations of those of ${\cal P}_{P_n}$. In particular, a $0$-face $(\emptyset,\alpha)$ of ${\cal P}_G$ corresponds to the vertex $\alpha(x)$ of $\Pi_n$, and, more generally, a face $(K,\alpha)$ of ${\cal P}_G$ corresponds (via the coset $T_K\alpha$) to the convex hull of the images of $\alpha(x)$ under the subgroup $T_K$ of the symmetry group $S_{n+1}$. Thus the graphicahedron ${\cal P}_{P_n}$ is isomorphic to the face-lattice of $\Pi_n$.
\end{proof}

Next we analyse the graphicahedra associated with more general kinds of connected graphs $G$. We are particularly interested in graphs with few vertices and edges. 

For the $n$-cycle $C_n$ with $n=3$ or $4$ it is not hard to verify directly that ${\cal P}_{C_n}$ is a vertex-transitive tessellation of the $2$-torus or $3$-torus, respectively. In particular, by Corollary~\ref{allsq}, ${\cal P}_{C_3}$ is a regular $3$-polytope isomorphic to the toroidal map $\{6,3\}_{(1,1)}$ (see \cite{cm}). The structure of ${\cal P}_{C_n}$ for general $n$ will be analyzed in the forthcoming 
paper \cite{del}; in fact, ${\cal P}_{C_n}$ is always a tessellation on the $(n-1)$-torus.

Next let $K_{1,n}$ be a star graph with $n$ edges emanating from a central vertex, so that $p=n+1$ and $q=n$. Then ${\cal P}_{K_{1,n}}$ is a regular $n$-polytope for each $n\geq 0$, with automorphism group $S_{n+1}\ltimes S_{n}$ (see Theorem~\ref{allsq}). Except for the $C_3$-graphicahedron, these are the only graphicahedra which are regular polytopes.

When $n=3$ we obtain the toroidal regular map ${\cal P}_{K_{1,3}}=\{6,3\}_{(2,2)}$. This occurs as the facet type of the regular $4$-polytope ${\cal P}_{K_{1,4}}$. In particular, ${\cal P}_{K_{1,4}}$ is the universal regular $4$-polytope $\{\{6,3\}_{(2,2)},\{3,3\}\}$ with automorphism group $S_{5}\times S_{4}$ and with $20$ toroidal facets each isomorphic to $\{6,3\}_{(2,2)}$ (see \cite[11C8]{McMS02}) and also \cite{grgcd,mshf,monw3}). Isomorphism of ${\cal P}_{K_{1,4}}$ with the universal polytope follows from that fact that both polytopes have groups of the same order; here, the semi-direct $S_{5}\ltimes S_{4}$ actually also is a direct product $S_{5}\times S_{4}$ (however, with different factors). More generally, for each $n\geq 0$, the regular $n$-polytope ${\cal P}_{K_{1,n}}$ is the facet type of the regular $(n+1)$-polytope ${\cal P}_{K_{1,n+1}}$, and the latter has $(n+1)(n+2)$ such facets; see also \cite[p.320]{stor} for a related polytope, possibly isomorphic to the graphicahedron of a star graph.

Clearly, a connected $(p,3)$-graph must necessarily be isomorphic to $P_3$, $C_3$ or $K_{1,3}$, so there are no graphicahedra of rank $3$ other than those already mentioned. This completely settles the case $q=3$. On the other hand, when $q=4$ there are two connected $(p,4)$-graphs in addition to $P_4$, $C_4$ or $K_{1,4}$ (see Figure~\ref{graph4}). The corresponding graphicahedra of rank $4$ are clearly not regular; their automorphism group is $S_5\ltimes C_2$ or $S_4\ltimes C_2$, respectively, with the factor $C_2$ arising from the graph symmetry.  For the graph on the left, the graphicahedron has twenty-five facets, ten permutahedra, five regular toroids $\{6,3\}_{(2,2)}$, and ten hexagonal prisms. For the graph on the right, the graphicahedron has seven facets, of which two are permutahedra, four are regular toroids $\{6,3\}_{(1,1)}$, and one is a regular toroid $\{6,3\}_{(2,2)}$.

\begin{figure}[h!]
\begin{center}
\includegraphics[width=10cm]{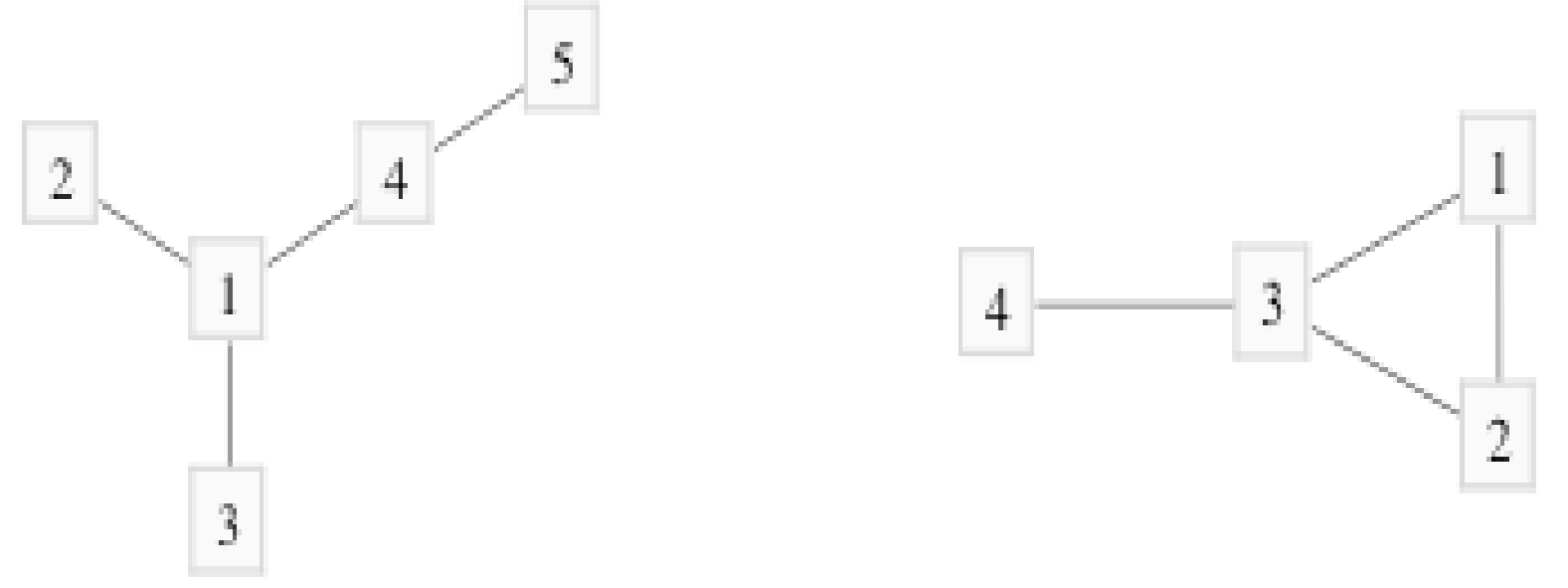}.
\caption{Graphs with four edges distinct from $P_4$, $C_4$ or $K_{1,4}$.}
\label{graph4}
\end{center}
\end{figure}

{\bf Acknowledgment:\/}  Some of the results described here were presented and discussed at the International Workshop on Polytopes and Symmetries held in March 2008 during the Coloquio Victor Neumann-Lara de Teoria de las Gr\'{a}ficas, Combinatoria y sus Aplicaciones at the Universidad Aut\'{o}noma de Zacatecas in Zacatecas, Mexico. We wish to extend our thanks to the participants, including in particular Javier Bracho, Dimitri Leemans, Barry Monson, Alen Orbanic, Daniel Pellicer, Toma\v{z} Pisanski and Asia Weiss.  

 \end{document}